\documentclass{amsart}

\usepackage[utf8]{inputenc} 
\usepackage[T1]{fontenc} 
\usepackage{ae} 
\usepackage{graphicx}
\usepackage[colorlinks,linkcolor=blue]{hyperref}

\numberwithin{equation}{section}

\newcommand{\bp}{\begin{picture}}
\newcommand{\bpn}{\begin{picture}(0,0)}
\newcommand{\ep}{\end{picture}}

\def\vir{\makebox(0,0){\rule{8pt}{8pt}}}
\def\cir{\circle*{0.23}}

\def\ci#1{%
{\bpn\put(0,0){\cir}\put(0,0.2){\makebox(0,0)[b]{$\scriptstyle #1$}}\ep}}
\def\vi#1{%
{\bpn\put(0,0){\vir}\put(0,0.2){\makebox(0,0)[b]{$\scriptstyle #1$}}\ep}}

\def\cirr#1{%
{\bpn\put(0,0){\cir}\put(0.2,0){\makebox(0,0)[l]{$\scriptstyle #1$}}\ep}}
\def\cill#1{%
{\bpn\put(0,0){\cir}\put(-0.2,0){\makebox(0,0)[r]{$\scriptstyle #1$}}\ep}}

\def\lijn{\line(1,0){1}}

\def\Bbb{\mathbb} \def\cal{\mathcal} 
\def\wt#1{\widetilde{#1}} 
\def\wl#1{\overline{#1}} 
\def\sier#1{{\cal O}_{#1}} 
\def\C{{\Bbb C}}    
\def\Q{{\Bbb Q}} \def\Z{{\Bbb Z}}
   \def\cO{{\cal O}}

\def\lra{\longrightarrow}
\def\fm{\mathfrak m}
 
\def\pabi#1/#2{\frac{\partial #1}{\partial x_{#2}}}

\DeclareMathOperator{\lcm}{lcm} 
 
\DeclareMathOperator{\weight}{wt} 
\def\ww#1={\weight(w_{#1})=} 
\def\wv#1={\weight(v_{#1})=} 
\def\wu#1={\weight(u_{#1})=} 
 
\newtheorem{theorem}{Theorem}[section] 
\newtheorem{prop}[theorem]{Proposition} 
\newtheorem{lemma}[theorem]{Lemma} 
\theoremstyle{definition} 
\newtheorem{defn}[theorem]{Definition} 
\newtheorem{remark}[theorem]{Remark} 
\newtheorem{example}[theorem]{Example} 

\begin{document} 
 
\title{Kulikov singularities}  
 
\author{Jan Stevens} 
\address{Department of Mathematical Sciences, Chalmers University of  
Technology and University of Gothenburg.  
SE 412 96 Gothenburg, Sweden} 
\email{stevens@chalmers.se}


\begin{abstract} 
In the study of normal surface singularities 
the relation between analytical and topological properties
and invariants of the singularity is a very rich problem.
This relation is particularly close for surface singularities
constructed from families of curves. We use these
Kulikov singularities to reexamine results of
N\'emethi-Okuma and Tomaru.
\end{abstract}

\maketitle 
 
\textit{\hfill Dedicated to the memory of Egbert Brieskorn.}
 
 \bigskip
 
\section*{Introduction}\label{s1} 
The first time I met Brieskorn was when I started my Ph.D. studies
in Leiden and he was spending some months there. Horst Knörrer
was then also working there.  Through his students Brieskorn
has influenced my career and work very much.  And of course
through his work, in the first place through his book with Knörrer
on plane algebraic curves \cite{BK}. This is a most remarkable
book, not only because of its value for money (Brieskorn negotiated
a price below DM 50) and its white cover, but mainly because its style
and contents.
Ever since curve singularities and algebraic curves have been central in
my work. 

Trying to describe singularities one may ask the question:
\medskip

\centerline{\it
Which discrete data are needed to know a singularity?}

\medskip
\noindent
One interpretation of ‘knowing a singularity’ is that we can
write down equations. As we
only have discrete data, such equations necessarily describe an equisingular family of singularities.

For plane curve singularities  there are very satisfactory answers
to the question, which can be found in Brieskorn's book \cite{BK}.
There is the link of the singularity,
which gives the embedded topology (without the embedding
one has only the number of components); another invariant
is the resolution graph. Since Brieskorn's work on the
exotic spheres as links of singularities it is realised that in high
dimension the abstract link contains not enough information.
In the surface case the situation is different. The topology of the
link, encoded in the resolution graph, is a strong invariant.
For rational and minimally elliptic singularities it determines
the equisingularity class. 
For higher geometric genus this is no longer the case and the
study of the relation between analytical and topological properties
and invariants of singularities is a very rich problem.

To have a strong relation we have to look at special  classes of
singularities. In the work of Neumann and Wahl
(for an overview see \cite{Wahl}) and of N\'emethi
two kind of restrictions are imposed, an analytical one, that the singularity is
$\Bbb Q$-Gorenstein, and a topological one, that the link is a rational homology sphere. 
Neumann and Wahl came even up with a way to write down equations
from the resolution graph, provided certain special numerical conditions are satisfied. The so called splice type equations describe a complete intersection singularity in a particular simple form, however  not for a singularity
with the original graph, but for its universal abelian cover (which is a finite cover due to  the rational homology sphere condition). 
In  a recent paper N\'emethi and Okuma \cite{NO} study 
which analytic structures can occur for a specific resolution graph,
giving details for an example already mentioned by N\'emethi
\cite{Nem}. One of the occurring structures is that of a Kodaira
or Kulikov singularity.

Kodaira singularities were introduced by Karras \cite{Kar0}, using a construction similar to the one earlier described by  Kulikov \cite{Kul}.
In my thesis \cite{JS} I introduced the term \textit{Kulikov singularities}.
The construction starts from a (degenerating)
1-parameter family $\pi \colon W\to D$ of curves of genus $g$. 
Let $\sigma \colon \smash{\wt W} \to W $ be the blow up of $W$ in $r$ points 
of the special fibre $W_0$, each point a smooth point on a component occurring with multiplicity $1$. Then the  strict transform 
of the special fibre   can be blown down to a singular point $p\in \wl W$.
By definition $(\wl W, p)$ is a Kulikov singularity.
 The study of properties of
such singularities reduces in two ways to the study of curves. 
The morphism $\pi$ descends to a function on the singularity, 
which defines a general hyperplane section. This curve singularity is
more accessible and invariants like its multiplicity and embedding dimension
determine the corresponding  invariants of the surface singularity.
The other occurrence of curves is by construction: the properties of the
central fibre, considered as curve of arithmetic genus $g$, are
essential.

Kulikov introduced his construction to give a uniform construction
of the unimodal and bimodal singularities. These are the simplest types
of minimally elliptic singularities. For higher genus  Kulikov singularities
should also be considered as simplest types. The generalisation of 
Laufer's minimally elliptic cycle \cite{Lau} is the characteristic cycle,
introduced by Karras for Kodaira singularities \cite{Kar} and in \cite{JS}
in general. Tomaru studied for which Brieskorn singularities
the characteristic cycle is equal to the fundamental cycle  \cite{Tom} .

Karras' work on Kodaira singularities of higher genus \cite{Kar} and
my work on Kulikov singularities \cite{JS} was never published.  When 
referred to, these singularities are mainly seen as singularities where there is
a function defining the fundamental cycle $Z$, which is moreover reduced at
components $E_i$ with $E_i\cdot Z<0$. In this paper I actually take
this as definition (see Definition \ref{Kdef}), 
being the shortest, but it is the construction using
a family of curves which gives a good understanding of the 
singularity.  As illustration  I treat the results of N\'emethi
and Okuma \cite{NO} and of Tomaru \cite{Tom} from  this point of view.

\section{Invariants of surface singularities}\label{sect_een}
The topological type of a normal complex
surface singularity is determined by and  determines
the resolution graph of the minimal good resolution. 
But a resolution graph 
can be defined for any resolution, not necessarily good.

\begin{defn}
Let $\pi\colon (M,E)\to (V,p)$ be a resolution of a surface 
singularity with exceptional divisor $E=\bigcup_{i=1}^r E_i$.
The \textit{resolution graph}  $\Gamma$ is a weighted graph 
with vertices 
corresponding to the irreducible components $E_i$. 
Each vertex has two weights,
the self-intersection $-b_i=E_i^2$, and the arithmetic genus $p_a(E_i)$, the second
traditionally written in square brackets and omitted if zero.
There is an edge between distinct vertices 
if the corresponding components $E_i$ and $E_j$ intersect,  
weighted with the intersection number
$E_i\cdot E_j$  (only written out if larger than one).
\end{defn}
Other definitions, which record more information, are possible:
one variant is to have an edge 
for 
each intersection point $P\in E_i\cap E_j$, with weight
the local intersection number $(E_i\cdot E_j)_P$. This is the more
common definition in the case that the intersections are transverse.

The classes of the curves $E_i$
form a preferred basis of $H:=H_2(M,\Z)$. 
Following
algebro-geometric tradition the elements of $H$ are called
\textit{cycles}. They are written as linear combinations of the $E_i$.
 The intersection form on $M$ gives  a negative definite
quadratic form on $H$. Let $K\in H^2(M,\Z)$ 
be the canonical
class. It can be written as rational cycle in $H_\Q=H\otimes \Q$ by solving
the adjunction equations $E_i\cdot(E_i+K)=2p_a(E_i)-2$.
The function $-\chi(A)=\frac12 A\cdot(A+K)$, $A\in H$, makes $H$ 
into a quadratic \textit{quadratic lattice},  in the sense of
\cite[1.4]{LW}. 
We prefer to work with the genus
$p_a(A)=1-\chi(A)$. Note that the genus function determines the
intersection form, as
\[
p_a(A+B)=p_a(A)+p_a(B)+A\cdot B -1\;.
\]
The data $(H,p_a)$ is equivalent to $(H, \{E_i\cdot E_j\},
\{p_a(E_i)\})$, encoded in the resolution graph $\Gamma$.

There are some important cycles on $E$, some of which only
depend on the quadratic lattice, while others depend on the
analytic structure. 

\begin{defn}\label{deffc}
The \textit{fundamental cycle} $Z$ is the is the smallest positive cycle 
such that $E_i\cdot Z\leq0$ for all $i$. The \textit{maximal ideal cycle}
$Z_\fm$ is the smallest cycle occurring as compact part of the divisor
of a function $f\in \fm_{(V,p)}$. The \textit{canonical cycle} $Z_K$ is the
rational cycle on $E$, which is numerically equivalent to the 
anticanonical class of the resolution $M$.
\end{defn}

We recall that the \textit{geometric genus} $p_g(V,p)$ is the dimension
of $R^1\pi_*\sier M$. This is equal to the maximal
value of $h^1(\sier D)$ over all positive cycles. In fact, there
is a unique minimal \textit{cohomological cycle} with this maximal
value (see \cite[4.8]{Reid}).  A topological lower bound for $p_g$ is
the arithmetic genus $p_a(V,p)$, which is the maximal
value of $p_a(D)$ over all positive cycles. The genus $p_a(Z)$
of the fundamental cycle is also a topological invariant of the singularity,
which is called the \textit{fundamental genus} $p_f(V,p)$  \cite{Tom}.

Obviously $p_f\leq p_a \leq p_g$, and all inequalities can be
strict; the easiest example with $p_a>p_f$ is the case of an 
irreducible exceptional curve $E$ of genus $g>1$ and self-intersection 
$-1$.  

\begin{defn} The \textit{characteristic cycle} $C$ of a nonrational
singularity is the smallest cycle
which realises the fundamental genus: it is the cycle $C\leq Z$
with  $p_a(C)=p_a(Z)$ and  $p_a(D)<p_a(C)$ for all cycles
$0<D<C$.
\end{defn}
This cycle is a generalisation of Laufer's minimally elliptc cycle
and its existence is proved in the same way.
It was first introduced by Karras for Kodaira singularities
\cite{Kar}. The general definition is in \cite{JS}; Tomaru also
introduced it under the name \textit{minimal cycle} \cite{Tom}.

\section{Kulikov singularities}\label{s2} 
In this section we introduce the Kulikov construction,
give some properties and discuss when the resulting singularity
is Gorenstein.

\begin{defn}\label{Kdef}
Let $(V,p)$ be a normal surface singularity with fundamental cycle
$Z$ on the minimal resolution. It 
is called a \textit{Kulikov singularity} if 
there exists a function $f\colon (V,p) \to (\C,0)$ with $(X,p)=(f^{-1}(0),p)$ 
a reduced curve singularity with divisor 
on the minimal resolution of the form $Z+\wt X$, such that the
strict (or proper) transform $\wt X$ of $X$ intersects 
the exceptional set $E$
transversally in smooth points on components having multiplicity
one in the fundamental cycle $Z$.
\end{defn}

Such singularities are the result of a construction first 
given by Kulikov [Kulikov], to describe the unimodal and bimodal singularities.
He starts from a (degenerating)
family $\pi \colon W\to D$ of curves of genus $g$. This 
is a proper morphism of a non-singular surface to a small disc.
The special fibre $W_0=\pi^{-1}(0)$  over $0\in D$ can be written as
$W_0=n_1C_1+\dots n_kC_k$, where the $C_i$ are the
irreducible components of this  fibre. 
The intersection matrix $(C_i\cdot C_j)$ is negative semi-definite.
Let $\sigma \colon \wt W \to W $ be the blow up of $W$ in $r$ points 
$q_1,\dots,q_r$, each a smooth point of a component $C_i$ which has
multiplicity $n_i=1$ in  $W_0$. We denote the strict transform 
of a component $C_i$ by $E_i$. Then the special fibre  $\wt W_0$ of
$\tilde \pi=\pi \circ \sigma$ can be written as
\[
\wt W_0=n_1E_1+\dots n_kE_k+\wt X_1+\dots +\wt X_r \; ,
\]
where the $\wt X_j$ are $(-1)$-curves.
Now the intersection matrix $(E_i\cdot E_j)$ is negative definite and
$E=\bigcup E_i$ can be blown down to a singular point $p\in \wl W$.

\begin{lemma}
Kulikov's construction results in a Kulikov singularity.
Conversely, every Kulikov singularity can be obtained by this construction. 
\end{lemma}

\begin{proof}
The construction yields the minimal resolution if there are no $(-1)$-curves
in the family $\pi \colon W\to D$ except possibly curves containing
a point $q_j$.  If there are other $(-1)$-curves we blow them down
without changing the resulting singularity. So we may assume that
$\wt W \to \wl W$ is the minimal resolution of the singularity $p\in
\wl W $. We write
$\wt W_0=Y+\wt X$  and have
to show that $Y$ is the fundamental cycle of the
singularity $(\wl W,p)$. 
We put   $Y=Z+D$ with $D$ an effective cycle
supported on $E$. Then $ D$ does not intersect 
$\wt X$, as each $\wt X_i$ intersects  $Y$ in a component with
multiplicity one. Now $0=D\cdot \wt W_0=D\cdot (Z+D+\wt X)= D\cdot Z+ 
D\cdot D\leq 0$, so $D\cdot D=0$ and therefore $D=0$.

Conversely, given a function $f\colon (V,p)\to (\C,0)$ with
divisor $Z+\wt X$ we compactify to a family of curves, following
Karras [1980,Thm 2.9]: in each point $q\in E\cap \wt X$ there are local
coordinates such that $f$ is given by $xy=0$, and $\wt X$ by $y=0$.
We glue the blow-up of the origin to it: with coordinates $(u,y)$
we have two charts, given by $(u,y)=(u,u\eta)=(xy,y)$. The glueing 
is by identifying the $(x,y)$ coordinates. Then $u=xy$ extends the 
function $f$.
\end{proof}

Kulikov singularities are a special case of Kodaira singularities,
defined by Karras \cite{Kar0,Kar}. 
In his construction it is allowed that points to
be blown up coincide: one blows up consecutively, and it is
allowed to blow up the strict transform of the fibre in a point
of intersection with a previously blown up curve.
Then the curve $(X,p)=(f^{-1}(0),p)$ is not necessarily
a reduced curve.

The advantage of the more strict definition of Kulikov singularities
is that the curve $(X,p)$ is a general hyperplane section. The function 
$f\colon (V,p) \to (\C,0)$ defines a smoothing of this curve with Milnor
number $\mu=2g+r-1$. The structure of the hyperplane section is often  
much easier to describe than that of the singularity itself. It allows
conclusion about the multiplicity and the embedding dimension of the
singularity.

An alternative description of the construction starts from a
\textit{minimal family} $\pi \colon W\to D$, meaning that $W$ does
not contain $(-1)$-curves. One then blows up points consecutively,
each time blowing up a point with multiplicity one in the special fibre.
In each stage a $(-1)$-curve intersects only one
other curve, so in the final surface the $(-1)$-curves are ends, and their
complement is connected. Write as before
$\wt W_0=Y+\wt X$ with $\wt X$ the union of the $(-1)$-curves.
Then the support of $Y$ can be blown down.

We have the following properties.
\begin{prop}\quad\label{propchar}
\begin{enumerate} 
\item For a Kulikov singularity
the maximal ideal cycle $Z_\fm$ is equal
to the fundamental cycle $Z$. 
\item The fundamental genus is equal to
the genus of the curves in the family used in the construction: 
$p_f(V,p)=g$. 
\item A rational singularity is Kulikov if and only if the fundamental cycle is 
reduced.
\item The characteristic cycle of a nonrational Kulikov singularity is the 
strict transform of the special fibre of the minimal family resulting in
the singularity. \label{char}
\end{enumerate} 
\end{prop}
\begin{proof}
Only the last property needs a proof. It suffices to consider the
case that the strict transform is the whole fundamental cycle.
Suppose that $C<Z$ and choose a computation sequence
$Z_j=Z_{j-1}+E_{i_j}$ from $Z_0=C$
to $Z_k=Z$.  As $p_a(Z_j)=p_a(Z)$ for all $j$, each $E_{i_j}$ 
is a smooth rational curve with $E_{i_j}\cdot Z_{j-1}=1$.
This holds in particular for the last one and therefore
$E_{i_k}\cdot Z=1+E_{i_k}^2<0$. This implies that $E_{i_k}$
has multiplicity one in the fundamental cycle and $E_{i_k}\cdot=
-E_{i_k}\cdot \wt X$. After blowing down $\wt X$ the strict transform
of $E_{i_k}$ has self-intersection $(-1)$, contradicting that the 
family we started from was a minimal family.
\end{proof}

To obtain a Gorenstein Kulikov singularity we have to perform the
construction in special points.
Let $\pi \colon W\to D$ be a minimal family of curves of genus $g$. The
relative dualising sheaf $\omega_{W/D}$ is isomorphic to $\Omega_W$. 
Let $(\omega)$
be the divisor of a global section.  It consists of an horizontal, 
non-compact part  $N$ and a divisor supported on the special fibre, determined
up to a multiple of this fibre.
Suppose that each component of $N$ intersects the special fibre
only  transversally in components of multiplicity one. 
Now we perform the Kulikov construction starting from the minimal family, blowing up at least 
these intersection points, in such a way that in the final family 
$\tilde \pi\colon \wt W\to D$ the pull back of $\omega$ has the 
same multiplicity $m$ along all $(-1)$-curves $X_i$, and that
the horizontal part of its divisor intersects the special fibre
only  in $\wt X$. Let $f=\tilde \pi^*(t)$, with $t$  a coordinate
function on $D$. Then the meromorphic two-form $f^{-m}\omega$ 
is holomorphic and nowhere zero on $U\setminus E$, $U$ a 
neighbourhood of $E$. Therefore the Kulikov singularity
is Gorenstein.

\begin{example}
We give an example of a 1-parameter family of weighted homogeneous
Gorenstein singularities $V_a$ such that $V_0$ is not Kulikov
but $V_a$ is Kulikov for $a\neq 0$.
It is the simplest of the series of examples of Brian\c con and Speder
of a family which is $\mu$-constant, but not $\mu^*$-constant \cite{BS}.

Consider 
\[
f_a(x,z,t)=z^3+azx^3+tx^4+t^9\;.
\]
The resolution graph is 
\[
\bp(1,1)(0,-0.5)
\put(0,0){\cir}\put(0,0.2){\makebox(0,0)[b]{$\scriptstyle -2$}}
\put(0,-0.2){\makebox(0,0)[t]{$\scriptstyle [3]$}}
\put(0,0){\lijn}
\put(1,0){\ci{-2}} 
\ep
\]
The  exceptional divisor on the minimal resolution is
$E=E_1+E_2$ with $E_1$ a curve of genus 3 with self-intersection
$-2$, and $E_2$  a rational $(-2)$-curve. 
The canonical model of $E_1$ is the plane quartic $\eta\zeta^3+a\zeta\xi^3+\xi^4+\eta^4$;
 this curve has a
flex in $P=(0:0:1)$, and the tangent $\eta=0$ intersects the curve in
$Q=(-a:0:1)$, so for $a=0$ there is a hyperflex. 
The normal bundle of
$E_1$ has $P+Q$ as divisor, and $E_2$ intersects $E_1$ in $Q$. The
general hyperplane section  has two branches for $a\neq
0$; the strict transform of one branch passes through $P$, and the other
intersects $E_2$ in a smooth point of $E$. For $a=0$ the curve is
irreducible, its strict transform passes through $P=Q=E_1\cap E_2$.

\[
\includegraphics[width=.8\textwidth]{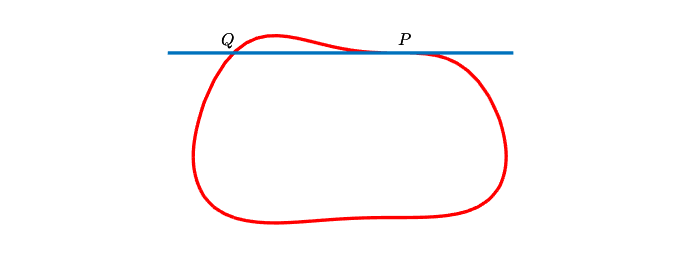} 
\]

To construct this singularity we start from the
trivial family $W=E_1\times D$. A canonical divisor 
is $3P\times D+ Q\times D$. After blowing up in $P\times \{0\}$ the multiplicity along the newly introduced exceptional divisor is
$4$.  Blowing  up in $Q\times \{0\}$ gives multiplicity $2$.
We blow up again in intersection point of special fibre and strict 
transform of section $Q \times D$, resulting in multiplicity $4$. 
By dividing with $t^4$ we see that the singularity is Gorenstein
with  $K=-4E_1-2E_2$. 
The functions  $t$, $x=t^2\xi/\eta$ and $z=t^3\zeta/\eta$ are
holomorphic on neighbourhood of $E$, giving $(t^3\zeta/\eta)^3+a(t^3\zeta/\eta)(t^2\xi/\eta)^3+t(t^2\xi/\eta)^4+t^9=0$; this formula works also 
for $a=0$. The blowing up can be done in family over a base $D\times A$,
with $a$ a coordinate on $A$. We first blow up in 
$P\times {0}\times A$, then in $Q$ as lying
on the strict transform of $C\times {0}\times A$ and then once again
in the intersection point with the strict transform of the appropriate section.
For $a=0$ this means that we blow up in a double point 
of the special fibre, which is not allowed in the Kulikov construction.
\end{example}

\section{The characteristic cycle of Brieskorn-Pham singularities} 

The simplest type of quasi-homogeneous hypersurface 
singularities has an equation, which is a sum of perfect powers,
and is usually called a Brieskorn-Pham polynomial.
We write in the surface case
\begin{equation}\label{pham}
x^a+y^b+t^c
\end{equation}
 with $a\leq b \leq c$.
It is well known how to get the resolution of this
surface singularity from the exponents $a$, $b$ and $c$ \cite{OW}.
The precise form is not important for us now.

\begin{lemma}
If $c\geq \lcm(a,b)$,
the Brieskorn-Pham singularity \eqref{pham} is a 
Kulikov singularity  of genus
$g=(\mu-r+1)/2$, where  $\mu=(a-1)(b-1)$ is the Milnor number 
of the curve singularity $x^a+y^b$ and $r=\gcd(a,b)$ is the
number of branches.
\end{lemma}

\begin{proof}
We construct the singularity with the Kulikov construction. We start
with an affine family of curves, whose equation is in fact given by 
a Brieskorn-Pham polynomial, but with lower exponent $c$.
Put $d= \lcm(a,b)$. Let $r=\gcd(a,b)$, then $d=\frac{ab}r$.
Consider the  family $\xi^a+\eta^b+t^{c-d}=0$ as family of affine plane
curves, parametrised by $t$,
 and complete it in the weighted projective space with weights 
$(\frac da,  \frac db,1)$. The homogeneous
equation is then $\xi^a+\eta^b+t^{c-d}w^d=0$. We resolve the 
singularity at the origin. We look at the chart $\xi=1$. There
the equation is $1+\bar\eta^b+t^{c-d} \bar w^d=0$, modulo the action
$\frac ad(\frac db,1)$. For $t=0$ we have $1+\bar\eta^b=0$,  so there
are indeed $\frac b{a/d}=r$ points  on the compactification of 
the special fibre. The coordinate transformation from $(\xi,\eta,1)$
coordinates to $(1,\bar \eta,\bar w)$ is 
$\xi=\bar w^{-\frac da}$, $\eta=\bar\eta\bar w^{-\frac db}$.
We blow up in the $r$ points at infinity on the special fibre.
The functions $x:=\xi t^{\frac da}$, $y:=\eta t^{\frac da}$ and $t$ are holomorphic
in a neighbourhood of the strict transform of the special fibre, and
generate the local ring of the Kulikov singularity. They satisfy
$x^a+y^b+t^c=0$.
\end{proof}

It follows that the family of curves obtained by resolving the 
singularity of $\xi^a+\eta^b+t^{c-d}w^d=0$
is not minimal if $c-d\geq d=\lcm(a,b)$.
Furthermore the resolution graph of $x^a+y^b+t^{c-d}$ 
is a subgraph of the resolution graph of $x^a+y^b+t^{c}$. 

\begin{prop}
Write $c=c_0+c_1d$ with $0\leq c_0<d$. The characteristic cycle
of the Brieskorn-Pham singularity \eqref{pham} has support
on the subgraph corresponding to the singularity
$x^a+y^b+t^{c_0+d}$ and is the fundamental cycle of that
singularity. In particular, the characteristic cycle is equal to the
fundamental cycle if and only if $d\leq c <2d$.
\end{prop}

\begin{proof}
If the family used in the construction above is not minimal,
one can blow down each component of the strict transform
of the affine curve $\xi^a+\eta^b=0$ and still have a family
of the same type. So the family is minimal if and only
$c-d<d$.
The result now follows from Proposition \ref{propchar}.(\ref{char}).
\end{proof}

The Proposition was proved by Tomaru \cite{Tom} using an  explicit description of the resolution of the singularity.
As to this resolution, we note that there
are $r$ chains of $c_1-1$
$(-2)$-curves from the characteristic cycle to
the components of $\wt X$.

\begin{remark}
The above result extends with the same proof to the 
case of Brieskorn complete intersections. A proof 
in the style of \cite{Tom} was given by Meng, Yuan and Wang
 \cite{MYW}.
\end{remark}

\section{Singularities with a specific resolution graph}  
A  recent paper N\'emethi and Okuma \cite{NO} concerns the problem of determining upper and
lower bounds for the geometric genus in terms of the 
resolution graph.
The Authors  study 
which analytic structures can occur for a specific resolution graph,
giving details for an example already mentioned by N\'emethi
\cite{Nem}. Here we rederive their results from our point of view.

The main  feature of the example is that the topological upper bound 
for $p_g$ is not realised. The maximal $p_g$ occurs for a non Gorenstein
Kulikov singularity and for a Gorenstein splice type singularity.

The singularity considered has an integral homology sphere
link. The resolution graph for the minimal good resolution
is:

\[
\bp(4,1.4)(0,-1.2)
\put(0,0){\ci{-3}}  \put(0,0){\lijn}
\put(1,0){\ci{-1}} \put(1,0){\lijn}
\put(1,0){\line(0,-1){1}} \put(1,-1){\cill{-2}}
\put(2,0){\vi{-13}}  \put(2,0){\lijn}
\put(3,0){\ci{-1}}  \put(3,0){\lijn}
\put(3,0){\line(0,-1){1}} \put(3,-1){\cirr{-2}}
\put(4,0){\ci{-3}}
\ep
\]
This graph satisfies the semigroup condition of Neumann and Wahl
\cite{NW} so there exist singularities of splice type with this graph,
with $p_g=3$.
The defining
equations of this complete intersection singularity have \lq leading\rq\
forms 
\begin{equation}\label{splice}
z_1^2z_2 + z_3^2+ z_4^3,\qquad z_1^3+ z_2^2+ z_4^2z_3\;.
\end{equation}

On the minimal resolution the exceptional
curve is an irreducible two-cuspidal rational curve, of
 self-intersection $-1$.
Therefore the resolution graph for the minimal resolution is simply:
\begin{equation}\label{genustwo}
\bp(0,1)
\put(0,0.5){\cir}\put(0,0.7){\makebox(0,0)[b]{$\scriptstyle -1$}}
\put(0,0.3){\makebox(0,0)[t]{$\scriptstyle [2]$}}
\ep
\end{equation}
with a possibly singular
central curve.
This is the same graph as when the exceptional divisor 
is a smooth  curve of genus two. We note that there exists
a Gorenstein Kulikov singularity with this graph, namely the
hypersurface $z^2=y^5+x^{10}$; it has  the maximal geometric
genus: $p_g=4$.

We first analyse the Gorenstein condition. On the 
minimal resolution  $M$
adjunction gives for the exceptional curve
that  $\omega_E=
\omega_{M}\otimes \sier E(E)$. 
The singularity is Gorenstein
if and only if $\omega_{M}=\sier{M}(-3E)$. This happens 
if and only if $\omega_E=\sier E(-2E)$,  that is, if the conormal
bundle of $E$ is a theta characteristic.

\begin{lemma}
A singularity with resolution graph \eqref{genustwo}
satisfies   $2\leq p_g\leq 4$.  If $p_g=4$   then it is a Gorenstein
Kulikov singularity. If $p_g=3$ it is either non Gorenstein
Kulikov of multiplicity 3 or a non Kulikov complete intersection.
\end{lemma}
\begin{proof}
To  analyse the possible values for $p_g$ we  look at a
computation sequence. Here one compares 
the different $\cO(-kE)$ via the short exact sequences
\[
0\lra \cO(-(k+1)E) \lra \cO(-kE)\lra \sier E(-kE) \lra 0
\]
As $H^1(\wt X, \cO(-3E))=0$
one gets the exact sequences
\[
0 \lra H^1(\wt X, \cO(-E)) \lra H^1(\wt X, \cO)
\lra H^1(E, \sier E) \lra 0
\]
\begin{multline*}
\qquad H^0(E, \sier E(-E)) \lra H^1(\wt X, \cO(-2E)) \lra \\
\lra 
H^1(\wt X, \cO(-E))
\lra H^1(E, \sier E(-E)) \lra 0 \qquad
\end{multline*}
and the isomorphism $H^1(\wt X, \cO(-2E))\cong H^1(E, \sier E(-2E))$.

This gives   $2\leq p_g\leq 4$. If $p_g=4$ 
then $\sier E(-2E)=\omega_E$,
so the singularity is Gorenstein. Moreover, the theta characteristic
is odd. Indeed, on a smooth genus two curve the divisor of a 
Weierstrass point is an odd theta characteristic. The Kulikov
construction starting from a trivial family and blowing just 
one Weierstrass point lying on the central fibre, yields the 
example $z^2=y^5+x^{10}$.

A two-cuspidal rational curve has only one theta characteristic,
which is even \cite{Har}. 
This can also be seen from the description of the pencil
with this special fibre in the list of Namikawa and Ueno \cite{NU}:
their example is $y^2=(x^3+t)((x-1)^3+t)$, and one sees
that three Weierstrass points come together in cusp. 
This shows that there cannot be 
a singularity with this exceptional divisor with $p_g=4$.
But any computation with the quadratic lattice $H$ cannot 
distinguish between such a curve and a smooth curve.

A non Gorenstein Kulikov singularity is obtained by blowing up
one smooth point of the special fibre; for a smooth curve
this point should not be a Weierstrass point.
By construction the general hyperplane section is a curve with Milnor fibre
of genus two, so $\delta=2$. The only irreducible
non Gorenstein curve singularity is the monomial curve
$(t^3,t^4,t^5)$.  Therefore the surface singularity has 
multiplicity $3$ and embedding dimension $4$.
In this case  $H^0(E, \sier E(-E))=\C$, so $p_g=3$.

If the singularity is Gorenstein, but not Kulikov, 
then $p_g=3$ and the curve $E$ has an even theta characteristic.
For a smooth $E$ there exists a quasi-homogeneous
singularity.
Let $y^2=f_6(x,\bar x)$ be a hyperelliptic curve $E$, and write
$f_6=PQ$ with $P$, $Q$ of degree 3. Consider the divisor $(P)=2D$,
with $D$ a divisor of degree 3 on $E$, consisting of three
Weierstrass points. Then $\sier E(D-K_E)$ is an even theta characteristic. 
The graded ring $\bigoplus H^0(E,\sier E(k(D-K_E)))$ is generated by
$z=xP$, $\bar z=\bar x P$, $w=yP$ and $v=P^2$. The equations
are then
\[
w^2=Q(z,\bar z), \qquad v^2= P(z,\bar z)\;.
\]
The singularity with  two-cuspidal curve as exceptional curve
is a superisolated complete intersection singularity. The graded
tangent cone is obtained in the same way as above, by taking 
$P=x^3$, $Q=\bar x^3$. We have to add terms of lowest degree to
make the singularity isolated, resulting in splice
diagram equations of the form \eqref{splice}:
\[
w^2=\bar z^3+vz^2, \qquad v^2= z^3+w\bar z^2\;.
\]
\end{proof}

Finally a quasi-homogeneous singularity with $p_g=2$
is obtained from a divisor $D-K_E$ with$D$ a general effective divisor
of degree 3 on a smooth curve $E$. 
The graded ring $\bigoplus H^0(E,\sier E(k(D-K_E)))$
has 7 generators. The same ring for the two-cuspidal rational
curve gives a weighted tangent cone of a singularity in $\C^7$.

\end{document}